\documentclass{lms}

\usepackage{amssymb}
\usepackage{amsmath}

\newtheorem{theorem}{Theorem}[section] 
\newtheorem{lemma}[theorem]{Lemma}     

\newtheorem{proposition}[theorem]{Proposition}
\newtheorem{definition}[theorem]{Definition}
\newtheorem{remark}[theorem]{Remark}

\newnumbered{assertion}{Assertion}    
\newnumbered{conjecture}{Conjecture}  
\newnumbered{hypothesis}{Hypothesis}
\newnumbered{note}{Note}
\newnumbered{observation}{Observation}
\newnumbered{problem}{Problem}
\newnumbered{question}{Question}
\newnumbered{algorithm}{Algorithm}
\newnumbered{example}{Example}
\newunnumbered{notation}{Notation} 

\title[$W^{2,p}$ estimates]
 {$W^{2,p}$ interior estimates of fully nonlinear elliptic equations} 

\author{Dongsheng Li and Kai Zhang}

\classno{35J60, 35B65 (primary), 35B20, 35D40 (secondary).}

\extraline{This research is supported by NSFC 11171266.}

\begin{document}
\maketitle

\begin{abstract}
In this paper, we generalize the $W^{2,p}$ interior estimates of fully nonlinear elliptic equations that were obtained by Caffarelli in \cite{Ca}.  The generalizations are carried out in two directions. One is that we relax the regularity requirement on the "constant coefficients" equations and the other one is that we broaden the range of $p$.
\end{abstract}

\section{Introduction}
\label{}
This paper deals with the following equation:
\begin{equation}\label{eq1.1}
 F(D^2u,x)=f(x)~~~~\mbox{in}~~B_1,
\end{equation}
where $B_1$ is the unit ball in $R^n~(n\geq 2)$ and $F:S(n)\times B_1\rightarrow R$ is uniformly elliptic, i.e., there exist $0<\lambda\leq \Lambda$ such that for $a.e.$ $x\in B_1$,
\begin{equation*}
    \lambda\|N\|\leq F(M+N,x)-F(M,x)\leq \Lambda\|N\|,~~~~\forall~M,N\in S(n),~N\geq 0,
\end{equation*}
where $S(n)$ denotes the set of $n\times n$ symmetric matrices and $\|N\|$ the spectral radius, and $N\geq 0$ means the nonnegativeness.

$W^{2,p}$ estimates will be obtained by using perturbation method. Roughly speaking, suppose that all solutions to $F(D^2u,x_0)=0$ have sufficient regularity for all $x_0\in B_1$ and the oscillation of $F$ in $x$ is small enough, then the desired regularity can be obtained for $F(D^2u,x)=f$ in $B_{1/2}$. This was done by Caffarelli in his celebrated work \cite{Ca} (see also Chapter 7 \cite{C-C}), where he required that $F(D^2u,x_0)=0$ satisfied $C^{1,1}$ interior estimates and obtained $W^{2,p}$ estimates of (\ref{eq1.1}) for $n<p<+\infty$. Here we will generalize this result to less of a requirement on the regularity of $F(D^2u,x_0)=0$ and a larger range of $p$.

In this paper, we will use $C^2-$viscosity solutions and $W^{2,p}-$viscosity solutions whose definitions can be found in many papers. For instance, see Definition 2.3 in \cite{C-C} for the former (called viscosity solutions there) and  Definition 2.1 in \cite{C-C-K-S} for the later (called $L^p-$ viscosity solutions there). In order to state our results clearly, we need the following definitions.

\begin{definition}\label{de3.1}
We say that $F(D^2w,x_0)$ has $C^{1,1}$ interior estimates with constant $c_e$ if for any $w_0\in C(\partial B_1)$, there exists a $C^2-$viscosity solution $w \in C^2(B_1)\cap C(\bar{B}_1)$ of
\begin{equation*}
 \left\{\begin{aligned}
 F(D^2w(x),x_0)&=0~~~~\mbox{in}~~B_1;\\
 w&=w_0~~~~\mbox{on}~~\partial B_1,
\end{aligned}
\right.
\end{equation*}
such that
\begin{equation*}\label{eqn3.1}
\|w\|_{C^{1,1}(\bar{B}_{1/2})}\leq c_e\|w_0\|_{L^{\infty}(\partial B_1)}.
\end{equation*}
\end{definition}

Weaker $W^{2,q}$ interior estimates requirement will be used in this paper:

\begin{definition}\label{de3.1}
We say that $F(D^2w,x_0)$ has $W^{2,q}$ interior estimates with constant $c_e$ if for any $w_0\in C(\partial B_1)$, there exists a $W^{2,q}-$viscosity solution $w \in W^{2,q}_{loc}(B_1)\cap C(\bar{B}_1)$ of
\begin{equation*}
 \left\{\begin{aligned}
 F(D^2w(x),x_0)&=0~~~~\mbox{in}~~B_1;\\
 w&=w_0~~~~\mbox{on}~~\partial B_1,
\end{aligned}
\right.
\end{equation*}
such that
\begin{equation}\label{eq3.1}
\|w\|_{W^{2,q}(B_{1/2})}\leq c_e\|w_0\|_{L^{\infty}(\partial B_1)}.
\end{equation}
\end{definition}

To measure the oscillation of $F$ in $x$, we need the following definition.

\begin{definition}\label{de3.2}
We define
\begin{equation*}
    \beta (x,x_0)=\beta_{F}(x,x_0)=\sup_{M\in S(n)\backslash\{0\}}\frac{|F(M,x)-F(M,x_0)|}{\|M\|}
\end{equation*}
and denote
\begin{equation*}
\beta(x):=\beta(x,0).
\end{equation*}
\end{definition}

In 1989, Caffarelli proved the following (Theorem 7.1\cite{C-C}, see also Theorem 1\cite{Ca}):
\begin{proposition}\label{prop1}
Let $u$ be a $C^2-$viscosity solution of (\ref{eq1.1}). Assume that $F(0,\cdot)\equiv0$ in $B_1$ and that $F(D^2w,x_0)$ has $C^{1,1}$ interior estimates with constant $c_e$ for any $x_0\in B_1$. Let $n<p<+\infty$ and suppose that $f\in L^p(B_1)$.

Then there exist positive constants $\beta_0$ and $C$ depending only on $n$, $\lambda$, $\Lambda$, $c_e$ and $p$, such that if
\begin{equation*}\label{}
     \left(\frac{1}{|B_r(x_0)|}\int_{B_r(x_0)}\beta(x,x_0)^{n}dx\right)^{\frac{1}{n}}\leq \beta_0,~~\forall~B_r(x_0)\subset B_1,
\end{equation*}
then $u\in W^{2,p}(B_{1/2})$ and
\begin{equation*}\label{}
    \|u\|_{W^{2,p}(B_{1/2})}\leq C\left(\|u\|_{L^{\infty}(B_1)}+\|f\|_{L^{p}}\right).
\end{equation*}
\end{proposition}

We will relax the condition that $F(D^2w,x_0)$ has $C^{1,1}$ interior estimates to that $F(D^2w,x_0)$ has $W^{2,q}$ interior estimates. Of course, we need a closer connection between (\ref{eq1.1}) and $F(\cdot,x_0)=0$ in the sense that $p_2>n$ for the integrability of $\beta(x,x_0)$. Actually, we have:

\begin{theorem}\label{thh1.2}
Let $n<p<p_1<+\infty$ and
\begin{equation*}
\frac{1}{p_1}+\frac{1}{p_2}=\frac{1}{n}.
\end{equation*}
Let $u$ be a $W^{2,p}-$viscosity solution of (\ref{eq1.1}). Assume that $F(0,\cdot)\equiv0$ in $B_1$ and that $F(D^2w,x_0)$ has $W^{2,p_1}$ interior estimates with constant $c_e$ for $a.e.$ $x_0\in B_1$. Suppose that $f\in L^p(B_1)$.

Then there exist positive constants $\beta_0$ and $C$ depending only on $n$, $\lambda$, $\Lambda$, $c_e$, $p$, $p_1$ and $p_2$, such that if
\begin{equation*}\label{eq1.2}
    \Bigg(\frac{1}{|B_r(x_0)|}\int_{B_r(x_0)}\beta(x,x_0)^{p_2}dx\Bigg)^{\frac{1}{p_2}}\leq \beta_0,~~\forall~B_r(x_0)\subset B_1,
\end{equation*}
then $u\in W^{2,p}(B_{1/2})$ and
\begin{equation*}\label{eq1.3}
    \|u\|_{{W^{2,p}(B_{1/2})}}\leq C\left(\|u\|_{{L^{\infty}}(B_1)}+\|f\|_{{L^{p}}(B_1)}\right).
\end{equation*}
\end{theorem}

\begin{remark}\label{re1.00} (i) Since $W^{2,\infty}(B_{1/2})=C^{1,1}(\bar B_{1/2})$, we see that
Proposition \ref{prop1} is a special case of Theorem \ref{thh1.2}.

(ii) It should be noted that the viscosity solutions of $F(D^2u,x_0)=0$ do not necessarily lie in $C^{1,1}$. In fact, Nadirashvili et al. formulated equations $F(D^2u)=0$ who have solutions in $W^{2,p}$ for some $n<p$ and not in $C^{1,1}$ (see \cite{N-V-3}- \cite{N-T-V}).

(iii) The condition $F(0,\cdot)\equiv0$ is not essential since we may consider $G(M,x)=F(M,x)-F(0,x)$.

(iv) Proposition \ref{prop1} requires that $F$ is continuous in $x$ whereas we don't need this assumption.
\end{remark}

To prove Proposition \ref{prop1} and Theorem \ref{thh1.2}, the Aleksandrov-Bakelman-Pucci maximum principle and Harnack inequality will be used essentially, where $f\in L^p$ with $p\geq n$ is needed.
Relying on Fok's results (Theorem 3.1 and Theorem 5.20 in \cite{Fo}), we see that the range of $p$ can be enlarged to $p\geq n-\epsilon_0$, where
$\epsilon_0>0$ depends only on $n$, $\lambda$ and $\Lambda$. Accordingly, we generalize Theorem \ref{thh1.2} to the following.

\begin{theorem}\label{th1.2}
Let $n-\varepsilon_0< p_0<p<p_1<+\infty$ and
\begin{equation*}
\frac{1}{p_1}+\frac{1}{p_2}=\frac{1}{p_0},
\end{equation*}
where $\varepsilon_0>0$ originates from Theorem 3.1 in \cite{Fo} and depends only on $n$, $\lambda$ and $\Lambda$.
Let $u$ be a $W^{2,p}-$viscosity solution of (\ref{eq1.1}). Assume that $F(0,\cdot)\equiv0$ in $B_1$ and that $F(D^2w,x_0)$ has $W^{2,p_1}$ interior estimates with constant $c_e$ for $a.e.$ $x_0\in B_1$. Suppose that $f\in L^p(B_1)$.

Then there exist positive constants $\beta_0$ and $C$ depending only on $n$, $\lambda$, $\Lambda$, $c_e$, $p$, $p_0$, $p_1$ and $p_2$ such that if
\begin{equation*}\label{eq1.2}
    \Bigg(\frac{1}{|B_r(x_0)|}\int_{B_r(x_0)}\beta(x,x_0)^{p_2}dx\Bigg)^{\frac{1}{p_2}}\leq \beta_0,~~\forall~B_r(x_0)\subset B_1,
\end{equation*}
then $u\in W^{2,p}(B_{1/2})$ and
\begin{equation*}\label{eq1.3}
    \|u\|_{{W^{2,p}(B_{1/2})}}\leq C\left(\|u\|_{{L^{\infty}}(B_1)}+\|f\|_{{L^{p}}(B_1)}\right).
\end{equation*}
\end{theorem}

 A priori $W^{2,p}$ estimates for $n-\varepsilon<p$ were obtained by Escauriaza \cite{Es} in 1993 where $\varepsilon>0$ can be traced back to \cite{C-F}. In 1996, Fok in \cite{Fo} generalized the Aleksandrov-Bakelman-Pucci maximum principle of (\ref{eq1.1}) to allowing $f\in L^p$ for $n-\varepsilon_0< p$. At almost the same time, Caffarelli et al. (see Theorem B.1 in \cite{C-C-K-S}) obtained analogous $W^{2,p}$ estimates for strong solutions instead of classical solutions in \cite{Es}. This allowed authors to treat equations in measurable functions space. Furthermore, \cite{C-C-K-S} generalized the equation to the full form, i.e., $F(D^2u,Du,u,x)=f$, where a structure condition was needed. In 1997, \'{S}wiech gave the $W^{2,p}$ estimates for $W^{2,p}-$viscostiy solutions for $p\geq n-\varepsilon_0$ (see Theorem 3.1 \cite{Sw}), but our method is different from his. Recently, Winter \cite{Wi} obtained the global estimates analogous to interior estimates in \cite{Sw}. We also note that Wang \cite{Wa} gave the $W^{2,p}$ estimates of fully nonlinear parabolic equations.

 It should be noted that all mentioned papers above depend on the $C^{1,1}$ estimates for equations with constant coefficients whereas we only need a weaker $W^{2,p_1}$ interior estimates.

This paper is organized as follows: We prepare some preliminaries in Section 2. Theorem \ref{th1.2} will be proved in Section 3 and it is clear that Theorem \ref{thh1.2} is an easy consequence of Theorem \ref{th1.2}. The main technique of the proof is borrowed from \cite{C-C}. That is, we use polynomials of degree 2 to touch solutions and estimate the decay of the measure of the set on which touching polynomials have large aperture. We use the following notations in this paper, many of which are standard.

\begin{notation}\label{}
1. $|x|$: the Euclidean norm of $x\in R^n$.\\
2. $B(x,r)$ or $B_r(x)$: a ball of radius $r$ centered at $x$ in $R^n$ and $B_r$: $=B(0,r)$.\\
3. $Q(x,r)$ or $Q_r(x)$: a cube of side-length $r$ centered at $x$ in $R^n$ and $Q_r$: $=Q(0,r)$.\\
4. $|A|$: the Lebesgue measure of a measurable set $A\subset R^n$.\\
5. $S(n)$: the set of $n\times n$ symmetric matrices.\\
6. $\|N\|$: the spectral radius of $N\in S(n)$.\\
7. $N^{+}$, $N^-$: the positive and the negative parts of $N\in S(n)$.\\
8. $\mbox{tr}(N)$: the trace of $N\in S(n)$.\\
9. $m(f)(x_0)=\sup_{r>0}\frac{1}{|B_r(x_0)|}\int_{B_r(x_0)}|f(x)|dx$, the maximal function of $f$.
\end{notation}

\section{Preliminaries}
For the reader's convenience, we collect some preliminaries related to fully nonlinear elliptic equations and viscosity solutions. In addition, we present here some results which will be used in the next section.

First, we introduce the Pucci extremal operators (see Section 2.2\cite{C-C} or \cite{Pu}). Let $0<\lambda\leq\Lambda$ and $N\in S(n)$ be given. We define
\begin{equation*}
 M^+(N)=M^+(N,\lambda,\Lambda)=\Lambda \mbox{tr}(N^+)-\lambda \mbox{tr}(N^-)
\end{equation*}
and
\begin{equation*}
M^-(N)=M^-(N,\lambda,\Lambda)=\lambda \mbox{tr}(N^+)-\Lambda \mbox{tr}(N^-).
\end{equation*}

Let $\underline S^q(\lambda,\Lambda,f)$ denote the set of $W^{2,q}-$viscosity subsolutions of $M^+\left(D^2u,\lambda,\Lambda\right)=f(x).$ Similarly, let $\bar S^q(\lambda,\Lambda,f)$ denote the set of $W^{2,q}-$viscosity supersolutions of $M^-\left(D^2u,\lambda,\Lambda\right)=f(x).$ We also define
\begin{equation*}
   \begin{aligned}
 S^q(\lambda,\Lambda,f)=\underline S^q(\lambda,\Lambda,f)\cap \bar S^q(\lambda,\Lambda,f).
 \end{aligned}
\end{equation*}

Next, since the idea of the proof of Theorem \ref{th1.2} is to use polynomials of degree 2 to touch the solutions, we present here some preliminaries about this technique.

\begin{definition}\label{de2.7}
A paraboloid is a polynomial in $(x_1,...,x_n)$ of degree 2. We call $P$ a paraboloid of opening $M$ if
\begin{equation*}
    P=l_0+l(x)\pm\frac{M}{2}|x|^2,
\end{equation*}
where $M$ is a positive constant, $l_0$ a constant and $l$ a linear function. $P$ is convex if we have $+$ in above equation and concave when we have $-$.
\end{definition}

\begin{definition}\label{de2.8}
Given two continuous functions $u$ and $v$ defined in an open set $\Omega$ and a point $x_0\in\Omega$, we say that $v$ touches $u$ by above at $x_0$ in $\Omega$ whenever
\begin{gather*}
u(x)\leq v(x),~~\forall~x\in \Omega,\\
u(x_0)=v(x_0).
\end{gather*}
We also have the analogous definition of touching by below.
\end{definition}

\begin{definition}\label{de2.9}
Let $\Omega\subset R^n$ be bounded and $u\in C(\Omega)$. For $x\in \Omega$, we define
\begin{equation*}
\begin{split}
\overline\Theta (u,\Omega)(x)=&\inf\{M: \text {there exists a convex paraboloid of opening $M$ }\\
&\text{that touches $u$ by above at $x$ in $\Omega$} \}.
\end{split}
\end{equation*}
Analogously, we define
\begin{equation*}
\begin{split}
\underline\Theta (u,\Omega)(x)=&\inf\{M: \text {there exists a concave paraboloid of opening $M$ }\\
&\text{that touches $u$ by below at $x$ in $\Omega$}\}.
\end{split}
\end{equation*}
Finally, we define $ \Theta (u,\Omega)(x)=\max\{\overline\Theta (u,\Omega)(x), \underline{\Theta} (u,\Omega)(x)\}$.
\end{definition}

$\Theta (u,\Omega)$ can be used to represent the second derivatives of $u$. Here, we have the following two results. The first was proved by Caffarelli and Cabr\'{e} (see Proposition 1.1\cite{C-C}):
\begin{proposition}\label{pr2.10}
Let $1<q\leq +\infty$ and $u$ be a continuous function in $\Omega$. Let $\varepsilon$ be a positive constant and define
\begin{equation*}
\Theta (u,\varepsilon)(x):=\Theta (u,\Omega\cap B_\varepsilon(x))(x),~~~~x\in\Omega.
\end{equation*}
Assume that $\Theta (u,\varepsilon)\in L^q(\Omega)$. Then $D^2u\in L^q(\Omega)$ and
\begin{equation*}
\|D^2u\|_{L^q(\Omega)}\leq 2n\|\Theta (u,\varepsilon)\|_{L^q(\Omega)}.
\end{equation*}
\end{proposition}
In fact, the converse is also true:
\begin{lemma}\label{le3.4}
Let $\Omega$ be a $C^{1,1}$ bounded domain and $n/2<q< +\infty$. If $u\in W^{2,q}(\Omega)$, then
\begin{equation}\label{eq3.4}
   \|\Theta (u,\Omega)\|_{L^q(\Omega)}\leq C\|u\|_{W^{2,q}(\Omega)},
\end{equation}
where $C$ depends only on $n$, $q$ and $\Omega$.
\end{lemma}
\begin{proof}
Let $\tilde u$ be the extension of $u$ to $R^n$. Take $q_1=(q+n/2)/2$ and $n/2<q_0<\min(n,q_1)$. Let $x_0\in\Omega$ be a Lebesgue point of $\tilde u$, $D\tilde{u}$ and $D^2\tilde{u}$. We define
\begin{equation*}\label{}
    h(x)=\tilde{u}(x)-\left(\tilde{u}(x_0)+D\tilde{u}(x_0)(x-x_0)+\frac{1}{2}(x-x_0)^{T}D^2\tilde{u}(x_0)(x-x_0)\right),~x\in R^n.
\end{equation*}
Let $r=|x-x_0|$, $q_0^*=nq_0/(n-q_0)$ and $(Dh)_{(x_0,r)}=\frac{1}{|B_r(x_0)|}\int_{B_r(x_0)}Dh$. Then by Morrey's inequality, we have
\begin{equation}\label{n3191}
\begin{aligned}
&|h(x)|=|h(x)-h(x_0)|\leq Cr\left(\frac{1}{|B_r(x_0)|}\int_{B_r(x_0)}|Dh|^{q_0^*}\right)^{1/q_0^*}\\
&\leq Cr\left(\frac{1}{|B_r(x_0)|}\int_{B_r(x_0)}|Dh(y)-(Dh)_{(x_0,r)}|^{q_0^*}\right)^{1/q_0^*}+Cr\cdot (Dh)_{(x_0,r)}.
\end{aligned}
\end{equation}
From Poincare's inequality, we have
\begin{equation}\label{n3192}
\begin{aligned}
&\left(\frac{1}{|B_r(x_0)|}\int_{B_r(x_0)}|Dh(y)-(Dh)_{(x_0,r)}|^{q_0^*}\right)^{1/q_0^*}\\
&\leq Cr\left(\frac{1}{|B_r(x_0)|}\int_{B_r(x_0)}|D^2h|^{q_0}dy\right)^{1/q_0}\\
&= Cr\left(\frac{1}{|B_r(x_0)|}\int_{B_r(x_0)}|D^2\tilde{u}(y)-D^2\tilde{u}(x_0)|^{q_0}dy\right)^{1/q_0}\\
&\leq Cr\left(\frac{1}{|B_r(x_0)|}\int_{B_r(x_0)}|D^2\tilde{u}(y)-D^2\tilde{u}(x_0)|^{q_1}dy\right)^{1/q_1}\\
&\leq Cr\left(\left(m(|D^2\tilde{u}|^{q_1})(x_0)\right)^{1/q_1}+|D^2\tilde{u}(x_0)|\right),
\end{aligned}
\end{equation}
where $m(g)$ denotes the maximal function of $g$. On the other hand, from Appendix C in \cite{C-C-K-S}, we have
\begin{equation}\label{n3193}
\begin{aligned}
|(Dh)_{(x_0,r)}|&\leq Cr\sup_{0<s\leq r}\frac{1}{|B_s(x_0)|}\int_{B_s(x_0)}|D^2\tilde{u}(y)-D^2\tilde{u}(x_0)|dy\\
&\leq Cr\left(m(|D^2\tilde{u}|)(x_0)+|D^2\tilde{u}(x_0)|\right).
\end{aligned}
\end{equation}

Hence, from (\ref{n3191}), (\ref{n3192}) and (\ref{n3193}), we have
\begin{equation*}
|h(x)|\leq Cr^2\left(\left(m(|D^2\tilde{u}|^{q_1})(x_0)\right)^{1/q_1}+m(|D^2\tilde{u}|)(x_0)+|D^2\tilde{u}(x_0)|\right).
\end{equation*}
Note that $\tilde u=u$ in $\Omega$, we have
\begin{equation*}
\begin{aligned}
u(x)\leq& u(x_0)+Du(x_0)(x-x_0)+\\
&\frac{C}{2}\left(\left(m(|D^2\tilde{u}|^{q_1})(x_0)\right)^{1/q_1}+m(|D^2\tilde{u}|)(x_0)+|D^2u(x_0)|\right)|x-x_0|^2,~~~~\forall~x\in \Omega.
\end{aligned}
\end{equation*}
Hence, for $a.e.$ $x_0\in\Omega$,
\begin{equation*}
\overline\Theta(u,\Omega)(x_0)\leq \frac{C}{2}\left(\left(m(|D^2\tilde{u}|^{q_1})(x_0)\right)^{1/q_1}+m(|D^2\tilde{u}|)(x_0)+|D^2u(x_0)|\right).
\end{equation*}
We have the corresponding inequality for $\underline \Theta(u,\Omega)$. Therefore,
\begin{equation*}
\begin{aligned}
\|\Theta (u,\Omega)\|_{L^q(\Omega)}&\leq C\left(\|m(|D^2\tilde{u}|^{q_1})\|_{L^{q/q_1}(R^n)}^{1/q_1}+
 \|m(|D^2\tilde{u}|)\|_{L^{q}(R^n)}+\|D^2u\|_{L^q(\Omega)}\right)\\
 &\leq C\|u\|_{W^{2,q}(\Omega)}.
\end{aligned}
\end{equation*}
\end{proof}

From the viewpoint of which polynomials touch a solution, we have the following definition.
\begin{definition}\label{de2.11}
Let  $\Omega\subset R^n$ be bounded and $u\in C(\Omega)$. For any $ M>0 $, we define
\begin{equation*}
\begin{split}
  \overline{G}_{M}(u,\Omega)=&\{x\in \Omega: \text{there exists a convex paraboloid of opening $M$}\\
  &\text{that touches $u$ by above at $x$ in $\Omega$}\}
\end{split}
\end{equation*}
and $\overline{A}_{M}(u,\Omega)=\Omega\backslash \overline{G}_{M}(u,\Omega).$ Analogously, we define
\begin{equation*}
\begin{split}
  \underline{G}_{M}(u,\Omega)=&\{x\in \Omega: \text{there exists a concave paraboloid of opening $M$}\\
  &\text{that touches $u$ by below at $x$ in $\Omega$}\}
\end{split}
\end{equation*}
and $\underline{A}_{M}(u,\Omega)=\Omega\backslash \underline{G}_{M}(u,\Omega).$ We finally define $G_M(u,\Omega)=\underline{G}_{M}(u,\Omega)\cap\overline{G}_{M}(u,\Omega)$ and $A_{M}(u,\Omega)=\Omega\backslash {G}_{M}(u,\Omega)$.
\end{definition}

\begin{remark}\label{re_330}
From above definitions, it is clear that
\begin{equation}\label{eq2.1}
\left\{\Theta(u,\Omega)> t\right\}\subset  A_t(u,\Omega)\subset\left\{\Theta(u,\Omega)\geq t\right\},~~~~\forall~t>0.
\end{equation}
\end{remark}

We also use the following Calder\'{o}n-Zygmund cube decomposition. Let $Q_1$ be the unit cube. We split it into $2^n$ cubes of half side. We do the same splitting with each one of these $2^n$ cubes and we iterate this process. The cubes obtained in this way are called dyadic cubes. If $Q$ is a dyadic cube different from $Q_1$, we say that $\tilde Q$ is the predecessor of $Q$ if $Q$ is one of the $2^n$ cubes obtained from dividing $\tilde Q$.
\begin{proposition}\label{pr2.13}
Let $A\subset B\subset Q_1$ be measurable sets and $0<\delta<1$ such that\\
(i) $|A|\leq \delta$,\\
(ii) If $Q$ is a dyadic cube such that $|A\cap Q|>\delta |Q|$, then $\tilde Q\subset B$.

Then $|A|\leq\delta|B|$.
\end{proposition}
\begin{proof}
See Lemma 4.2\cite{C-C}.
\end{proof}

We also need the following proposition whose proof is left to the reader.
\begin{proposition}\label{pr2.14}
Let $g$ be a nonnegative and measurable function defined in $\Omega$ and $\mu_g$ be its distribution function, i.e.,
\begin{equation*}
\mu_g(t)=\left|\left\{x\in\Omega: g(x)>t\right\}\right|,~~~~t>0.
\end{equation*}
Let $\eta>0$ and $M>1$ be constants. Then, for $0<q<+\infty$,
\begin{equation}\label{eq2.2}
g\in L^q(\Omega)\Longleftrightarrow \sum_{k\geq 1}M^{qk}\mu_g(\eta M^k)=:S<+\infty
\end{equation}
and
\begin{equation*}
C^{-1}S\leq\|g\|^{q}_{L^q(\Omega)}\leq C\left(|\Omega|+S\right),
\end{equation*}
where $C>0$ is a constant depending only on $\eta$, $M$ and $q$.
\end{proposition}

\section{$W^{2,p}$ interior estimates}
In this section, we will prove Theorem \ref{th1.2}. By scaling and covering arguments, we only need to prove the following:
\begin{theorem}\label{th3.3}
Let $p, p_0,p_1$ and $p_2$ be the same as in Theorem \ref{th1.2}. Let $u$ be a $W^{2,p}-$viscosity solution of
\begin{equation*}\label{1}
    F(D^2u,x)=f(x)~~~~\mbox{in}~~B_{8\sqrt{n}}.
\end{equation*}
Suppose that $F(0,\cdot)\equiv0$ and  $F(D^2w,x_0)$ has $W^{2,p_1}$ interior estimates with constant $c_e$ for any $x_0\in B_{8\sqrt{n}}$. Assume that
\begin{gather*}\label{4}
\|u\|_{L^{\infty}(B_{8\sqrt{n}})}\leq 1,~~ \|f\|_{L^{p}(B_{8\sqrt{n}})}\leq \varepsilon
\intertext{and}
    \left(\frac{1}{|B_r(x_0)|}\int_{B_r(x_0)}\beta(x,x_0)^{p_2}dx\right)^{\frac{1}{p_2}}\leq \varepsilon,~~\forall~B_r(x_0)\subset B_{8\sqrt{n}}.
\end{gather*}
Then $u\in W^{2,p}(B_{1/2})$ and
\begin{equation*}\label{eq3.2}
    \|u\|_{{W^{2,p}(B_{1/2})}}\leq C,
\end{equation*}
where $\varepsilon$ and $C$ are positive constants depending only on $n$, $\lambda$, $\Lambda$, $c_e$, $p$, $p_0$, $p_1$, and $p_2$.

Furthermore, combining with Proposition \ref{pr2.10}, (\ref{eq2.1}) and Proposition \ref{pr2.14}, we only need to prove
\begin{equation}\label{eq3.3}
    \sum_{k\geq 1}M^{pk}|A_{M^k}(u,B_{1/2})|\leq C,
\end{equation}
for some constants $M$ and $C$ depending only on $n$, $\lambda$, $\Lambda$, $c_e$, $p_1$, $p_2$ and $p$.
\end{theorem}

The outline of the proof of (\ref{eq3.3}) follows that in \cite{C-C}. First, we prove that $|A_t|$ has a power decay in $t$ for $u\in S^q(f)$. Second, we use an approximation to accelerate the power decay corresponding to $F(D^2u)=f$. From now on, unless otherwise stated, $p, p_0,p_1$ and $p_2$ are fixed.

\begin{lemma}\label{le3.5}
Let $\Omega$ be a bounded domain such that $B_{6\sqrt{n}}\subset\Omega$ and $u$ be continuous in $\Omega$. Assume that $\|u\|_{L^{\infty}(\Omega)}\leq 1$ and $u\in S^q(\lambda,\Lambda,f)$ in $B_{6\sqrt{n}}$ with $\| f\|_{L^q(B_{6\sqrt{n}})}\leq 1$ for some $q> n-\varepsilon_0$. Then
\begin{equation}\label{eq3.5}
    |A_t(u,\Omega)\cap Q_1|\leq C_1t^{-\mu},~~\forall~t>0,
\end{equation}
where $C_1$ and $\mu$ are positive constants depending only on $n$, $\lambda$, $\Lambda$ and $q$.
\end{lemma}
\begin{proof}
We adopt the idea from Lemma 4 in \cite{Es}. It follows from Corollary 3.10\cite{C-C-K-S} that there exists a unique $W^{2,q}-$viscosity solution $v\in W^{2,q}_{loc}(B_{6\sqrt{n}})\cap C(\bar B_{6\sqrt{n}})$ of
\begin{equation*}
\left\{\begin{aligned}
M^{-}(D^2v)=f^+~~~~\mbox{in}~~B_{6\sqrt{n}},\\
v=0~~~~\mbox{on}~~\partial B_{6\sqrt{n}}.
\end{aligned}\right.
\end{equation*}
By the generalized Aleksandrov-Bakelman-Pucci maximum principle (see Theorem 3.1 and Remark 3.2\cite{Fo}), we have
\begin{equation*}
    0\geq v \geq -C\|f^+\|_{L^q(B_{6\sqrt{n}})}\geq -C.
\end{equation*}

Clearly, $M^-(D^2(u-v))\leq 0~~\mbox{in}~B_{6\sqrt{n}}$ in the viscosity sense. Extend $v$ continuously in $\Omega$ such that $|v|\leq C$. Hence, $\frac{u-v}{1+C}\leq 1$ in $\Omega$. Then by Lemma 7.8\cite{C-C}, we have
\begin{equation}\label{n13172}
    |\underline A_t(u-v,\Omega)\cap Q_1|\leq Ct^{-\mu},
\end{equation}
where $\mu$ depends only on $n$, $\lambda$ and $\Lambda$. Similarly, since $M^-(-D^2v)=-M^+(D^2v)\leq -f^+\leq 0$, we have
\begin{equation}\label{n13173}
    |\underline A_t(-v,\Omega)\cap Q_1|\leq Ct^{-\mu}.
\end{equation}

Next, we show that there exists a constant $C$ such that, for any $t\geq 1$,
\begin{equation}\label{n1317}
\underline A_{Ct}(v,\Omega)\cap Q_1 \subset \underline A_{t}(-v,\Omega)\cup \{x\in Q_1:m(f^q)(x)> t^q\}\cup N,
\end{equation}
where $N$ is a set of measure zero. Equivalently, we show that
\begin{equation}\label{n1317.3}
\underline G_{Ct}(v,\Omega)\cap Q_1\cup N \supset \underline G_{t}(-v,\Omega)\cap \{x\in Q_1:m(f^q)(x)\leq t^q\}.
\end{equation}

On the one hand, since $v$ is differentiable almost everywhere,  for $a.e.$ $x\in \underline G_t(-v,\Omega)$, we have
\begin{equation}\label{eq_n1}
v(y)\leq v(x)+Dv(x)(y-x)+\frac{t}{2}|y-x|^2,~~~~\forall~y\in\Omega.
\end{equation}
In particular, $w(y):=v(x)+Dv(x)(y-x)+2tr^2-v(y)$ is nonnegative and $M^+(D^2w)=-f^+$ in $B_{2r}(x)$ wherever $B_{2r}(x)\subset\Omega$. By the Harnack inequality (see Theorem 5.20\cite{Fo}), we have
\begin{equation*}
\sup_{B_r(x)}w\leq C\left(\inf_{B_r(x)}w+r^{2-n/q}\|f\|_{L^q\left(B_r(x)\right)}\right).
\end{equation*}
On the other hand, since $m(f^q)(x)\leq t^q$, we have $\|f\|_{L^q\left(B_r(x)\right)}\leq t r^{n/q}$. Hence,
\begin{equation*}
w(y)\leq\sup_{B_r(x)}w\leq C\left(\inf_{B_r(x)}w+tr^2\right)\leq C\left(w(x)+tr^2\right)=Ctr^2,~~~~\forall~y\in B_r(x),
\end{equation*}
or
\begin{equation}\label{eq_n2}
v(y)\geq v(x)+Dv(x)(y-x)-C|y-x|^2,~~~~\forall~y\in B_1(x).
\end{equation}

Taking $y=x-Dv(x)/|Dv(x)|\sigma$ in (\ref{eq_n1}) where $\sigma$ is chosen such that $y\in\partial B_{6\sqrt{n}}$ , and noting that $|v|\leq C$ and $t\geq 1$, we get
\begin{equation*}
|Dv(x)|\leq Ct.
\end{equation*}
Therefore, combining with (\ref{eq_n2}), we have
\begin{equation*}\label{}
v(y)\geq v(x)+Dv(x)(y-x)-Ct|y-x|^2,~~~~\forall~y\in\Omega.
\end{equation*}
Hence, $x\in \underline G_{Ct}(v,\Omega)\cap Q_1$, i.e., (\ref{n1317.3}) holds. Finally, from (\ref{n13172}), (\ref{n13173}) and (\ref{n1317}), we have
\begin{equation*}
\begin{aligned}
|\underline A_{2Ct}(u,\Omega)\cap Q_1|\leq |\underline A_{Ct}(u-v,\Omega)\cap Q_1|+|\underline A_{Ct}(v,\Omega)|\leq Ct^{-\mu}~~~~\forall~t>0.
\end{aligned}
\end{equation*}
By the similar argument, we have
\begin{equation*}
\begin{aligned}
|\overline A_{2Ct}(u,\Omega)\cap Q_1|\leq Ct^{-\mu}~~~~\forall~t>0.
\end{aligned}
\end{equation*}
This implies (\ref{eq3.5}).
\end{proof}

\begin{remark}\label{re3.3}
(\ref{eq3.5}) implies that $u\in W^{2,\delta}$ for some $\delta>0$. This was first proved by Lin \cite{Lin}, where $q\geq n$ is needed.
\end{remark}

We will use the following approximation lemma to accelerate the power decay of $A_t$.

\begin{lemma}\label{le3.6}
Let $0<\varepsilon<1$ and $u$ be a $W^{2,p_0}-$viscosity solution of
\begin{equation*}\label{}
    F(D^2u,x)=f(x)~~~~in~~B_{8\sqrt{n}}.
\end{equation*}
Suppose that $F(0,\cdot)\equiv 0$ and $F(D^2w,0)=0$ has $W^{2,p_1}$ interior estimates with constant $c_e$. Assume that
\begin{equation*}
\| u\|_{L^{\infty}(B_{8\sqrt{n}})}\leq 1,\mbox{ and }\|\beta \|_{L^{p_2}(B_{7\sqrt{n}})}=\|\beta(\cdot,0) \|_{L^{p_2}(B_{7\sqrt{n}})}\leq \varepsilon.
\end{equation*}

Then there exists $h\in W^{2,p_{1}}(B_{6\sqrt{n}})\cap C(\bar B_{7\sqrt{n}})$ such that $\|h\|_{W^{2,p_{1}}(B_{6\sqrt{n}})}\leq c(n)c_e$ (for a constant $c(n)$ depending only on $n$) and
\begin{equation}\label{eq3.6}
    \|u-h\|_{L^{\infty}(B_{6\sqrt{n}})}+\|\varphi\|_{L^{p_0}(B_{6\sqrt{n}})}\leq C_2(\varepsilon^{\gamma}+\|f\|_{L^{p_0}(B_{8\sqrt{n}})}),
\end{equation}
where $\varphi(x)=f(x)-F(D^2h(x),x)$, and $C_2$ and $0<\gamma<1$ are positive constants depending only on $n$, $\lambda$, $\Lambda$, $c_e$, $p_0$, $p_1$ and $p_2$.
\end{lemma}

\begin{proof}
Let $h\in W_{loc}^{2,p_{1}}(B_{7\sqrt{n}})\cap C(\bar B_{7\sqrt{n}})$ be the $W^{2,p_1}-$viscosity solution of
\begin{equation*}\label{}
    \left\{\begin{aligned}
        F(D^2h,0)&=0~~~~\mbox{in}~~B_{7\sqrt{n}},\\
        h&=u~~~~\mbox{on}~~\partial B_{7\sqrt{n}}.
    \end{aligned}\right.
\end{equation*}
Using the $W^{2,p_{1}}$ interior estimates and a covering argument, we have
\begin{equation*}\label{}
    \|h\|_{W^{2,p_{1}}(B_{6\sqrt{n}})}\leq c(n)c_e\|u\|_{L^{\infty}(B_{7\sqrt{n}})}\leq c(n)c_e.
\end{equation*}
By interior H\"{o}lder estimates (see Theorem 5.21\cite{Fo}), we know that $u\in C^{\alpha}(\bar{B}_{7\sqrt{n}})$ and
\begin{equation}\label{eq3.7}
    \|u\|_{C^{\alpha}(\bar{B}_{7\sqrt{n}})}\leq C\left(1+\|f\|_{L^{p_0}(B_{8\sqrt{n}})}\right),
\end{equation}
where $0<\alpha<1$ and $C>0$ depend only on $n$, $\lambda$, $\Lambda$ and $p_0$. From the interior H\"{o}lder estimates, we easily get the following global H\"{o}lder estimates (see Proposition 4.12 and Proposition 4.13 in \cite{C-C}):
\begin{equation}\label{eq3.8}
    \|h\|_{C^{\alpha/2}(\bar{B}_{7\sqrt{n}})}\leq C\|u\|_{C^{\alpha}(\bar{B}_{7\sqrt{n}})}\leq C(1+\|f\|_{L^{p_0}(B_{8\sqrt{n}})}),
\end{equation}
where $C>0$ depends only on $n$, $\lambda$ and $\Lambda$.

Let $0<\delta<1$ and $x_0\in B_{7\sqrt{n}-\delta}$. Then $B_{\delta}(x_0)\subset B_{7\sqrt{n}}$ and we apply the scaled version of $W^{2,p_{1}}$ interior estimates in $B_{\delta}(x_0)$ to $h-h(x_1)$ where $x_1\in\partial B_{\delta}(x_0)$; we get
\begin{equation*}\label{}
    \|D^2h\|_{L^{p_1}(B_{\delta/2}(x_0))}\leq c_e\delta^{\frac{n-2p_1}{p_1}}\|h-h(x_1)\|_{L^{\infty}(\partial B_{\delta}(x_0))}\leq C \delta^{\frac{n-2p_1}{p_1}}\delta^{\frac{\alpha}{2}}(1+\|f\|_{L^{p_0}(B_{8\sqrt{n}})}),
\end{equation*}
by (\ref{eq3.8}). By a standard covering argument, we have
\begin{equation}\label{add_12_10_1}
   \|D^2h\|_{L^{p_1}(B_{7\sqrt{n}-\delta})}\leq C \delta^{\frac{n-2p_1}{p_1}}\delta^{\frac{\alpha}{2}}\delta^{-n}(1+\|f\|_{L^{p_0}(B_{8\sqrt{n}})}).
\end{equation}
From the definition of $\beta$, we have
\begin{equation*}\label{}
    |F(D^2h(x_0),x_0)|\leq \beta(x_0)\|D^2h(x_0)\|~~\mbox{for}~a.e.~x_0\in B_{7\sqrt{n}-\delta}.
\end{equation*}
Thus,
\begin{equation}\label{add_12_10_2}
    \|F(D^2h(x),x)\|_{L^{p_0}(B_{7\sqrt{n}-\delta})}\leq \|\beta\|_{L^{p_2}(B_{7\sqrt{n}-\delta})}\|D^2h\|_{L^{p_1}(B_{7\sqrt{n}-\delta})},
\end{equation}
recall $\frac{1}{p_1}+\frac{1}{p_2}=\frac{1}{p_0}$. On the other hand, since $u-h=0$ on $\partial B_{7\sqrt{n}}$, we have, by (\ref{eq3.7}) and (\ref{eq3.8}), that
\begin{equation}\label{add_12_10_3}
    \|u-h\|_{L^{\infty}(\partial B_{7\sqrt{n}-\delta})}\leq C\delta^{\frac{\alpha}{2}}\left(1+\|f\|_{L^{p_0}(B_{8\sqrt{n}})}\right).
\end{equation}

Finally, since $u-h\in S^p(\lambda/n,\Lambda,\varphi)$ in $B_{7\sqrt{n}}$ where $\varphi(x)=f(x)-F(D^2h(x),x)$, combining with (\ref{add_12_10_1}), (\ref{add_12_10_2}) and (\ref{add_12_10_3}), we conclude that
\begin{equation*}
\begin{split}
&\|u-h\|_{L^{\infty}(B_{7\sqrt{n}-\delta})}+\|\varphi\|_{L^{p_0}(B_{7\sqrt{n}-\delta})}\\
&\leq \|u-h\|_{L^{\infty}(\partial B_{7\sqrt{n}-\delta})} +C\|\varphi\|_{L^{p_0}(B_{7\sqrt{n}-\delta})}\mbox{ (by the ABP estimates)}\\
&\leq C\delta^{\frac{\alpha}{2}}\left(1+\|f\|_{L^{p_0}(B_{8\sqrt{n}})}\right)+ C\|f\|_{L^{p_0}(B_{7\sqrt{n}-\delta})}+C\|F(D^2h(x),x)\|_{L^{p_0}(B_{7\sqrt{n}-\delta})}\\
\end{split}
\end{equation*}
Take $\delta=\varepsilon^{p_1/(np_1+2p_1-n)}$ to get (\ref{eq3.6}) with $\gamma=\frac{p_1\alpha}{2(np_1+2p_1-n)}$ (note that $B_{6\sqrt{n}}\subset B_{7\sqrt{n}-\delta}$).
\end{proof}

\begin{lemma}\label{le3.7}
Take $0<\bar\varepsilon<1$ and $M_1>1$ satisfying that
\begin{equation}\label{add_eq3.2}
    M_1^{p_1}\bar\varepsilon\geq 2\tilde{C}~~~~\mbox{and}~~~~(c_nM_1)^{p}\bar\varepsilon\leq 1/2,
\end{equation}
where $\tilde{C}$ is a constant which will be determined below and depends only on $n$, $c_e$ and $p_1$, and $c_n>1$ is a constant depending only on n which will be determined in (\ref{add_eq3.1}).

Let $\Omega$ be a bounded domain such that $B_{8\sqrt{n}}\subset\Omega$ and $u\in C(\Omega)$ be a $W^{2,p_0}-$viscosity solution of $F(D^2u,x)=f(x)$ in $B_{8\sqrt{n}}$. Suppose that $F(0,\cdot)\equiv0$ and $F(D^2w,0)=0$ has $W^{2,p_1}$ interior estimates with constant $c_e$. Assume that
\begin{equation*}
\begin{aligned}
&\|u\|_{L^{\infty}(B_{8\sqrt{n}})}\leq 1,~-|x|^2\leq u(x)\leq |x|^2~in~\Omega\backslash B_{6\sqrt{n}},\\
&~\|f\|_{L^{p_0}(B_{6\sqrt{n}})}\leq\varepsilon_1~and~\|\beta\|_{L^{p_2}(B_{7\sqrt{n}})}=\|\beta(\cdot,0)\|_{L^{p_2}(B_{7\sqrt{n}})}\leq \varepsilon_1,
\end{aligned}
\end{equation*}
where $0<\varepsilon_1<1$ depends only on $n$, $\lambda$, $\Lambda$, $c_e$, $p_0$, $p_1$, $p_2$ and $\bar\varepsilon$.

Then
\begin{equation}\label{eq3.9}
    |G_{M_1}(u,\Omega)\cap Q_1|\geq 1-\bar\varepsilon,
\end{equation}
\end{lemma}

\begin{proof}
Take $0<\varepsilon_1<1$ to be chosen later. Let $h$ be given by Lemma \ref{le3.6} applied with $\varepsilon_1$. Recall that $h\in C(\bar{B}_{6\sqrt{n}})$ and we extend $h$ outside $B_{6\sqrt{n}}$ continuously such that $h=u$ in $\Omega\backslash B_{7\sqrt{n}}$ and $\|u-h\|_{L^{\infty}(\Omega)}=\|u-h\|_{L^{\infty}(B_{6\sqrt{n}})}$. Since $\|h\|_{L^{\infty}(B_{6\sqrt{n}})}\leq \|u\|_{L^{\infty}(B_{7\sqrt{n}})} \leq 1$ (by maximum principle), we have $\|u-h\|_{L^{\infty}(\Omega)}\leq 2$ and hence, $-2-|x|^2\leq h(x)\leq 2+|x|^2$ in $\Omega\backslash B_{6\sqrt{n}}$. Combining with Lemma \ref{le3.4}, (\ref{eq2.1}) and $\|h\|_{W^{2,p_1}(B_{6\sqrt{n}})}\leq C$ where $C$ depends only on $n$ and $c_e$, we have
\begin{equation}\label{n319}
|A_t(h,\Omega)\cap Q_1|\leq C_3t^{-p_1},
\end{equation}
where $C_3$ depends only on $n$, $c_e$ and $p_1$.

Consider
\begin{equation*}
    w=\frac{u-h}{2C_2\varepsilon_1^{\gamma}},
\end{equation*}
where $C_2$ and $\gamma$ are the constants in (\ref{eq3.6}). It is easy to check, using Lemma \ref{le3.6}, that $w$ satisfies the hypothesis of Lemma \ref{le3.5} in $\Omega$. We apply Lemma \ref{le3.5} and get
\begin{equation*}\label{}
    |A_t(w,\Omega)\cap Q_1|\leq C_1t^{-\mu},~~\forall~t>0.
\end{equation*}
Thus, $|A_t(u-h,\Omega)\cap Q_1|\leq C\varepsilon_1^{\gamma\mu}t^{-\mu}$, for a constant $C$ depending only on $n$, $\lambda$, $\Lambda$, $c_e$, $p_0$, $p_1$ and $p_2$.

Therefore, combining with (\ref{n319}), we have
\begin{eqnarray*}
  |A_{2t}(u,\Omega)\cap Q_1|&\leq&|A_t(u-h,\Omega)\cap Q_1|+|A_t(h,\Omega)\cap Q_1|\\
  &\leq& C\varepsilon_1^{\gamma\mu}t^{-\mu}+C_3t^{-p_1}.
\end{eqnarray*}
Taking $\tilde{C}=2C_32^{p_1}$, $t=M_1/2$ and $\varepsilon_1$ small enough, then we have proved the lemma.
\end{proof}

\begin{lemma}\label{le3.8}
Let $\Omega$ be a bounded domain such that $B_{8\sqrt{n}}\subset \Omega$ and $u\in C(\Omega)$ be a $W^{2,p_0}-$viscosity solution of $F(D^2u,x)=f(x)$ in $B_{8\sqrt{n}}$. Suppose that $F(0,\cdot)\equiv0$ and $F(D^2w,0)=0$ has $W^{2,p_1}$ interior estimates with constant $c_e$. Assume that $\|f\|_{L^{p_0}(B_{8\sqrt{n}})}\leq \varepsilon_1$, $\|\beta\|_{L^{p_2}(B_{7\sqrt{n}})}=\|\beta(\cdot,0)\|_{L^{p_2}(B_{7\sqrt{n}})}\leq \varepsilon_1$. Then
\begin{equation}\label{eq3.10}
   G_1(u,\Omega)\cap Q_3\neq\emptyset\Rightarrow |G_M(u,\Omega)\cap Q_1|\geq 1-\bar\varepsilon,
\end{equation}
where $\bar\varepsilon$ and $\varepsilon_1$ are the same as in Lemma \ref{le3.7}, and $M=c_nM_1$ (see (\ref{add_eq3.2})).
\end{lemma}
\begin{proof}
Let $x_1\in G_1(u,\Omega)\cap Q_3$. It follows that there is an affine function $L$ such that
\begin{equation*}\label{}
-\frac{1}{2}|x-x_1|^2\leq u(x)-L(x)\leq \frac{1}{2}|x-x_1|^2~~~~\mbox{in}~\Omega.
\end{equation*}
Define
\begin{equation}\label{add_eq3.1}
    v=\frac{u-L}{c_n},
\end{equation}
where $c_n$ is a large number depending on $n$ such that $\|v\|_{L^{\infty}(B_{8\sqrt{n}})}\leq1$ and
\begin{equation*}
-|x|^2\leq v(x)\leq |x|^2~~\mbox{in}~~\Omega\backslash B_{6\sqrt{n}}.
\end{equation*}
Apply Lemma \ref{le3.7} to $v$, which is the $W^{2,p_0}-$viscosity solution of
\begin{equation*}
    \frac{1}{c_n}F(c_{n}D^2v,x)=\frac{f(x)}{c_n}~~~~\mbox{in}~B_{8\sqrt{n}},
\end{equation*}
and get $|G_{M_{1}}(v,\Omega)\cap Q_1|\geq 1-\bar\varepsilon$. Hence $|G_{M}(u,\Omega)\cap Q_1|=|G_{c_nM_1}(u,\Omega)\cap Q_1|=|G_{M_1}(v,\Omega)\cap Q_1|\geq 1-\bar\varepsilon.$
\end{proof}

\begin{lemma}\label{le3.9}
Under the hypothesis of Theorem \ref{th3.3} where $\varepsilon$ is small enough and depends only on $n$, $\lambda$, $\Lambda$, $c_e$, $p$, $p_0$, $p_1$ and $p_2$, extend $f$ by zero outside $B_{8\sqrt{n}}$ and let, for $k=0,1,2,...,$
\begin{equation*}
    \begin{split}
      A &=A_{M^{k+1}}(u,B_{8\sqrt{n}})\cap Q_1, \\
      B &=\left(A_{M^{k}}(u,B_{8\sqrt{n}})\cap Q_1\right)\cup\left\{x\in Q_1: m(f^{p_0})(x)\geq (c_3M^k)^{p_0}\right\}.
    \end{split}
\end{equation*}
Then
\begin{equation*}
    |A|\leq \bar\varepsilon|B|,
\end{equation*}
Where $M>1$ and $0<\bar\varepsilon<1$ are the same as in Lemma \ref{le3.8}, and $c_3>0$ depends only on $n$, $\lambda$, $\Lambda$ and $c_e$.
\end{lemma}
\begin{proof}
We use Proposition \ref{pr2.13} with $\delta=\bar\varepsilon$. Clearly, $A\subset B\subset Q_1$. Without loss of generality, we assume that $F(D^2w,0)=0$ has $W^{2,p_1}$ interior estimates. Since $\|u\|_{L^{\infty}(B_{8\sqrt{n}})}\leq 1$ and $\|\beta\|_{L^{p_2}(B_{7\sqrt{n}})}\leq c(n,p_2)\varepsilon$, Lemma \ref{le3.7} applied with $\Omega=B_{8\sqrt{n}}$, gives that
\begin{equation*}
|G_{M^{k+1}}(u,B_{8\sqrt{n}})\cap Q_1|\geq |G_{M}(u,B_{8\sqrt{n}})\cap Q_1|\geq |G_{M_1}(u,B_{8\sqrt{n}})\cap Q_1|\geq 1-\bar\varepsilon
\end{equation*}
by taking $\varepsilon$ small enough such that $c(n,p_2)\varepsilon\leq\varepsilon_1$ (We will always take $\varepsilon$ small enough so that we can apply Lemma \ref{le3.7} and Lemma \ref{le3.8}.). Hence, $|A|\leq \bar\varepsilon$. It remains to check that if $Q=Q_{1/2^i}(x_0)$ is a dyadic cube of $Q_1$ such that
\begin{equation}\label{eq3.11}
    |A_{M^{k+1}}(u,B_{8\sqrt{n}})\cap Q|=|A\cap Q|>\bar\varepsilon Q
\end{equation}
then $\tilde Q\subset B$ (recall that $\tilde Q$ is the predecessor of $Q$).

First, we assume that $F(D^2w,x_0)=0$ has $W^{2,p_1}$ interior estimates. Suppose $\tilde Q\nsubseteq B$ and take $x_1$ such that
\begin{equation}\label{eq3.12}
    x_1\in \tilde Q\cap G_{M^k}(u,B_{8\sqrt{n}})
\end{equation}
and
\begin{equation}\label{eq3.13}
m(f^{p_0})(x_1)< (c_3M^k)^{p_0}.
\end{equation}

Consider the transformation
\begin{equation*}
    x=x_0+\frac{1}{2^i}y,~~~~y\in Q_1,~x\in Q
\end{equation*}
and the function
\begin{equation*}
    \tilde u(y)=\frac{2^{2i}}{M^k}u\left(x_0+\frac{1}{2^i}y\right).
\end{equation*}
Let us check that $\tilde{u}$ satisfies the hypothesis of Lemma \ref{le3.8} with $\Omega$ replaced by $\tilde{\Omega}$, the image of $\Omega$ under the translation above; here $\Omega=B_{8\sqrt{n}}$.

Since $B_{8\sqrt{n}/2^i}\subset B_{8\sqrt{n}}$, we have that $B_{8\sqrt{n}}\subset\tilde\Omega$. Clearly $\tilde u(y)$ is a $W^{2,p_0}-$viscosity solution of $G(D^2\tilde u(y),y)=\tilde f(y)$ in $B_{8\sqrt{n}}$ where
\begin{equation*}
    G(D^2w,y)=\frac{1}{M^k}F(M^kD^2w,x_0+\frac{1}{2^i}y)
\end{equation*}
and
\begin{equation*}
    \tilde f(y)=\frac{1}{M^k}f(x_0+\frac{1}{2^i}y).
\end{equation*}
Also we have that $G(0,y)=F(0,x)\equiv0$ and $G(D^2w,0)=\frac{1}{M^k}F(M^kD^2w,x_0)$ has $W^{2,p_1}$ interior estimates with constant $c_e$. Next, since
\begin{equation*}
    \beta_G(y)=\beta_G(y,0)=\beta_F(x,x_0).
\end{equation*}
and $B_{7\sqrt{n}/2^i}(x_0)\subset B_{8\sqrt{n}}$, our assumptions on $\beta(x,x_0)$ imply that $\|\beta_G\|_{L^{p_2}(B_{7\sqrt{n}})}\leq c(n,p_2)\varepsilon$.

Note that $|x_1-x_0|\leq 2\sqrt{n}/2^{i}$ implies $B_{8\sqrt{n}/2^i}(x_0)\subset Q_{20\sqrt{n}/2^i}(x_1)$. Hence
\begin{equation*}
    \|\tilde{f}(y)\|^{p_0}_{L^{p_0}(B_{8\sqrt{n}})}=\frac{2^{in}}{M^{kp_0}}\int_{B_{8\sqrt{n}/2^i}(x_0)}|f(x)|^{p_0}dx\leq\frac{2^{in}}{M^{kp_0}}\int_{Q_{20\sqrt{n}/2^i}(x_1)}|f(x)|^{p_0}dx\leq c(n)c_3^{p_0}\leq \varepsilon_1^{p_0},
\end{equation*}
by (\ref{eq3.13}); we have taken $c_3$ small enough such that the last inequality is true. Finally, by (\ref{eq3.12}), $G_1(\tilde u,\tilde \Omega)\cap Q_3\neq\emptyset$. Then Lemma \ref{le3.8} gives
\begin{equation*}
    |G_M(\tilde u,\tilde \Omega)\cap Q_1|\geq 1-\bar\varepsilon=(1-\bar\varepsilon)|Q_1|,
\end{equation*}
which implies
\begin{equation*}
    |G_{M^{k+1}}( u, B_{8\sqrt{n}})\cap Q|\geq(1-\bar\varepsilon)|Q|.
\end{equation*}
This is a contradiction with (\ref{eq3.11}).

If $F(D^2u,x_0)$ doesn't have $W^{2,p_1}$ interior estimates, then for any $\delta>0$, we take a cube $Q_l(\tilde x_0)\subset Q$ such that $|Q\backslash Q_l(\tilde x_0)|<\delta$ and $\tilde Q\subset Q_{4l}(\tilde x_0)$. Then we can process the proof as above with noting that $G_1(u,\Omega)\cap Q_4\neq \emptyset$ also implies $|G_M(u,\Omega)\cap Q_1|\geq 1-\bar\varepsilon$ (see (\ref{eq3.10})). Since $\delta$ is arbitrary, we draw the conclusion.
\end{proof}

Now, we give the main result Theorem \ref{th1.2}'s

\begin{proof}
Recall that we only need to prove (\ref{eq3.3}).
We define, for $k\geq 0$,
\begin{equation*}
    \alpha_k=|A_{M^k}(u,B_{8\sqrt{n}})\cap Q_1|,~\beta_k=|\{x\in Q_1:m(f^{p_0})(x)\geq (c_3M^k)^{p_0}\}|.
\end{equation*}
By Lemma \ref{le3.9},
\begin{equation*}
\alpha_{k+1}\leq \bar\varepsilon(\alpha_k+\beta_k).
\end{equation*}
Hence
\begin{equation}\label{eq3.14}
    \alpha_k\leq \bar\varepsilon^k+\sum_{i=0}^{k-1}\bar\varepsilon^{k-i}\beta_i.
\end{equation}
Since $f^{p_0}\in L^{p/p_0}( B_{8\sqrt{n}})$, we have that $m(f^{p_0})\in L^{p/p_0}( B_{8\sqrt{n}})$ and
\begin{equation*}
    \|m(f^{p_0})\|_{L^{p/p_0}( B_{8\sqrt{n}})}\leq C\|f\|^{p_0}_{L^p( B_{8\sqrt{n}})}\leq C.
\end{equation*}
Hence, by Proposition \ref{pr2.14},
\begin{equation}\label{eq3.15}
\sum_{k\geq 0}^{}M^{pk}\beta_k\leq C.
\end{equation}
Finally, combining with (\ref{add_eq3.2}), (\ref{eq3.14}) and (\ref{eq3.15}), we obtain (recall $M=c_nM_1$):
\begin{equation*}
    \begin{split}
       \sum_{k\geq 1}^{}M^{pk}\alpha_k &\leq \sum_{k\geq 1}^{}(\bar\varepsilon M^p)^k+\sum_{k\geq 1}^{}\sum_{i=0}^{k-1}\bar\varepsilon^{k-i}M^{p(k-i)}M^{pi}\beta_i \\
         &\leq \sum_{k\geq 1}^{}2^{-k}+\left(\sum_{i\geq 0}^{}M^{pi}\beta_i\right)\left(\sum_{j\geq 1}2^{-j}\right)\leq C.
     \end{split}
\end{equation*}
where $C$ is a constant depending only on $n$, $\lambda$, $\Lambda$, $c_e$, $p_0$, $p_1$, $p_2$ and $p$.
\end{proof}



\affiliationone{
   Dongsheng Li and Kai Zhang\\
   School of Mathematics and Statistics\\
   Xi'an Jiaotong University\\
   Xi'an 710049\\
   China\\
   \email{lidsh@mail.xjtu.edu.cn\\
   zkzkzk@stu.xjtu.edu.cn}}
\end{document}